\documentclass{article}[12pt]

\usepackage{amsthm,amssymb,amsmath}
\usepackage[noblocks]{authblk}

\usepackage{amsmath}
\usepackage{amsopn}
\usepackage{amssymb}
\usepackage{latexsym}
\usepackage{amsfonts}
\usepackage{amsthm}
\usepackage[all]{xy}
\usepackage{tikz}
\usepackage{verbatim}
\usepackage{graphicx}

\newtheorem{Theorem}{Theorem}
\newtheorem{Lemma}[Theorem]{Lemma}
\newtheorem{Example}[Theorem]{Example}
\newtheorem{Remark}[Theorem]{Remark}

\newtheorem{Proposition}[Theorem]{Proposition}

\newcommand{\N}{\mathbb{N}}

\newcommand{\Z}{\mathbb{Z}}

\newcommand{\mc}[1]{\mathcal{#1}} 

\newcommand{\mt}[1]{\text{#1}}

\usepackage[all]{xy}

\begin{document}
\title{Omitting parentheses from the cyclic notation}
\author[1]{Mahir Bilen Can}
\author[2]{Yonah Cherniavsky}

\affil[1]{Department of Mathematics, Tulane University, New Orleans, USA; mcan@tulane.edu}
\affil[2]{Department of Computer Science and Mathematics, Ariel University, Israel; yonahch@ariel.ac.il}

\date{August 23, 2013}
\maketitle

\begin{abstract}
The purpose of this article is to initiate a combinatorial study of the Bruhat-Chevalley ordering
on certain sets of permutations obtained by omitting the parentheses from their standard cyclic notation.
In particular, we show that these sets form a bounded, graded, unimodal, rank-symmetric
and EL-shellable posets. Moreover, we determine the homotopy types of the associated order complexes.
\end{abstract}

\section{Introduction}

Let $n$ be a positive integer, and let $S_n$ denote the symmetric group of permutations on the set $[n]:=\{1,\dots, n\}$.
A {\em fixed-point-free involution} is a permutation $\pi \in S_n$ such that $\pi \circ \pi = id$ and $\pi(i) \neq i$ for all $i\in [n]$.
In their interesting paper~\cite{DS}, Deodhar and Srinivasan define and study
an analog of the Bruhat-Chevalley ordering on the set of all fixed-point-free involutions of $S_{2n}$.
In~\cite{UV}, Upperman and Vinroot investigate the natural extension of this partial ordering on the set of products of $n$
disjoint cycles of length $m$ in $S_{mn}$.
In this note we present a construction generalizing both of these works.
We achieve this by studying a self-map on $S_n$, which is in some sense very classical.

The {\em standard cyclic form} of a permutation $\pi \in S_n$ is the representation of $\pi$ as a product of disjoint cycles
$$
\pi=\left(a_{1,1},\dots,a_{k_1,1}\right)\left(a_{1,2},\dots,a_{k_2,2}\right)\cdots\left(a_{1,m},\dots,a_{k_m,m}\right),
$$
where $a_{1,1}<a_{1,2}<\cdots<a_{1,m}$ and $a_{1,j}<a_{i,j}$ for every $1\leqslant j\leqslant m$ and $2\leqslant i\leqslant k_j$.
Here, contrary to the commonly used convention, we do not suppress cycles of length one from our notation.

Recall that the {\em one-line notation} for $x \in S_n$ is defined by
$x=\left[a_1, a_2,\dots, a_n\right]$, where $x(j)=a_j$ for $1\leqslant j\leqslant n$.
We define the mapping $\Omega: S_n \rightarrow S_n$ as follows:
write $x\in S_n$ in the standard form and then omit the parentheses.
We view the resulting word as a permutation of $S_n$ written in one-line notation.
For example, let $x=(1,2)(3)(4,5)$ be a permutation given in its standard cyclic form.
Then $\Omega(x)=[1,2,3,4,5]$, which is the identity permutation in $S_5$.

It is clear that $\Omega$ is not injective, however, there are interesting subsets of $S_n$
on which its restriction is one-to-one.
Recall that a {\em composition} of $n$ is a sequence $\lambda = (\lambda_1,\dots, \lambda_k)$
of positive integers such that $\sum \lambda_i = n$. In this case, it is customary to write $\lambda \vDash n$.
We define the {\em composition type} of a permutation $x \in S_n$ to be the composition obtained from
the standard cyclic form of $x$ by considering the lengths of its cycles.
For example, if $x=(1,2)(3)(4,5)\in S_5$, then its composition type is $\lambda=(2,1,2)$.
We denote by $A_\lambda \subset S_n$ the set of all permutations of composition type
$\lambda \vDash n$.

In the works of Deodhar and Srinivasan~\cite{DS}, and Upperman and Vinroot~\cite{UV}
the authors use a certain restriction of our map $\Omega$ composed with the inverse map $w \mapsto w^{-1}$.
Both of these papers use $\phi$ to denote their corresponding map.
In~\cite{DS}, $\phi$ operates on the set of fixed-point-free involutions
and in ~\cite{UV} $\phi$ operates on the set of products of disjoint $m$-cycles.
Here, we apply $\Omega$ to arbitrary $A_\lambda$ for $\lambda \vDash n$,
and investigate the ``Bruhat-Chevalley ordering'' on the resulting set of permutations, which we denote by
$C_\lambda:=\Omega\left(A_\lambda\right)$.
To continue, let us recall the definition of this important partial ordering.

The {\em inversion number} of a permutation $x=[a_1,\dots, a_n] \in S_n$ is the cardinality of the set of inversions
\begin{equation}\label{E:invpermutation}
inv(x) := |\{ (i,j):\ 1\leq i < j \leq n,\ x(i)>x(j) \}|.
\end{equation}
The {\em Bruhat-Chevalley ordering} on $S_n$ is the transitive closure of the following covering relations:
$x= [a_1,\dots, a_n]$ is covered by $y=[b_1,\dots,b_n]$, if  $inv(y)=inv(x)+1$ and
\begin{enumerate}
\item $a_k=b_k$ for $k\in \{ 1,\dots,\widehat{i},\dots,\widehat{j},\dots,n\}$ (hat means omit those numbers),
\item $a_i=b_j$, $a_j=b_i$, and $a_i<a_j$.
\end{enumerate}

We are now ready to give a brief overview of our paper and state our main results.
In the next section, we introduce more notation and provide preliminary results for our proofs.

Our first main result, which prove in Section \ref{S:three} is about the gradedness of $C_\lambda$'s.
Recall that a finite poset is called {\em graded}, if all maximal chains are of the same length.

\vspace{.25cm}
\noindent
\textbf{Theorem A.}
Given a composition $\lambda \vDash n$, with respect to the restriction of the Bruhat-Chevalley
ordering the set $C_\lambda$ is a graded poset with a minimum and a maximum.
\vspace{.25cm}

Let $q$ be a variable, and for a subset $S\subset S_n$, let $G_S(q)$ denote its ``length generating function''
$G_S(q) := \sum_{x\in S} q^{inv(x)}$. The $q$-analog of a natural number $n$ is the polynomial
$[n]_q := 1+q +\cdots + q^{n-1}$. Our second main result, which prove in Section \ref{S:four} is

\vspace{.25cm}
\noindent
\textbf{Theorem B.}
The length generating function of $C_\lambda$ is of the form
$G_{C_\lambda}(q) =[i_1]_q [i_2]_q \cdots [i_r]_q$ for a suitable sequence
$2 \leq i_1 < \cdots < i_r \leq n-1$ determined by $\lambda$.
\vspace{.25cm}

\noindent
As a consequence, we see that the length generating functions $G_{C_\lambda}(q)$ are palindromic,
and therefore, the posets $C_\lambda$ are unimodal and rank-symmetric.

To state our third main result and its important corollary (our Theorem 4),
we need to recall a definition of the notion of lexicographic shellability of a poset:
A finite graded poset $P$ with a maximum and a minimum element is called {\em EL-shellable}
({\em lexicographically shellable}), if there exists a map $f=f_{\varGamma}: C(P) \rightarrow \varGamma$ between the
set of covering relations $C(P)$ of $P$ into a totally ordered set $\varGamma$ satisfying
\begin{enumerate}
\item in every interval $[x,y] \subseteq P$ of length $k>0$ there exists a unique saturated chain
$\mathfrak{c}:\ x_0=x < x_1 < \cdots < x_{k-1} < x_k=y$ such that the entries of the sequence
\begin{align}\label{JordanHolder}
f(\mathfrak{c}) = (f(x_0,x_1), f(x_1,x_2), \dots , f(x_{k-1},x_k))
\end{align}
is weakly increasing.
\item The sequence $f(\mathfrak{c})$ of the unique chain
$\mathfrak{c}$ from (1) is the smallest among all sequences of the form
$(f(x_0,x_1'), f(x_1',x_2'), \dots , f(x_{k-1}',x_k))$, where $x_0 < x_1' < \cdots < x_{k-1}' < x_k$.
\end{enumerate}

For Bruhat-Chevalley ordering, in an increasing order of generality, the articles~\cite{E},~\cite{P}, and~\cite{BW} show that
$S_n$ is a lexicographically shellable poset. In Section~\ref{S:three} we prove that

\vspace{.25cm}
\noindent
\textbf{Theorem C.}
For all compositions $\lambda \vDash n$, the posets $C_\lambda$ are EL-shellable.
\vspace{.25cm}


The {\em order complex} $\Delta(P)$ of a poset $P$ is the abstract simplical complex consisting
of all chains from $P$. An important consequence of EL-shellability is that the associated order complex
has the homotopy type of wedge of spheres or balls.
For example, when $P=S_n$ (with respect to Bruhat-Chevalley ordering), the order complex of $P$ triangulates
a sphere of dimension $n(n-1)/2$.
In our final main result, which prove in Section \ref{S:five}, we obtain

\vspace{.25cm}
\noindent
\textbf{Theorem D.}
Let $\lambda\vDash n$ be a composition.
\begin{enumerate}
\item If $\lambda=(n)$ or $\lambda=(n-1,1)$, then the poset $C_\lambda$ is isomorphic to $S_{n-1}$.
In this case it is just a copy of $S_{n-1}$ embedded into $S_n$.
Similarly, if $\lambda_1=\lambda_2=\cdots=\lambda_{k-1}=1$ and $\lambda_k=n-k+1$ or $\lambda_1=\lambda_2=\cdots=\lambda_{k-2}=1$, $\lambda_{k-1}=n-k$, $\lambda_k=1$ (like $(1,1,1,4,1)$), then the poset $C_\lambda$ is isomorphic to $S_{n-k}$.
Therefore, in these cases, $\Delta(C_\lambda)$ triangulates a sphere.
\item In all other cases the order complex of $C_\lambda$ triangulates a ball.
\end{enumerate}
\vspace{.25cm}

We finish our paper with final comments and future questions in Section \ref{S:six}.

\vspace{.5cm}
\textbf{Acknowledgements.} The first author is partially supported by the Louisiana Board of Regents Research Development grant.

\section{Preliminaries}\label{S:preliminaries}

Recall that in a poset $P$, an element $y$ is said to {\em cover} another element $x$, if $x < y$ and if $x \leq z \leq y$ for some $z\in P$,
then either $z=x$ or $z=y$. In this case, we write $x \leftarrow y$. Given $P$, we denote by $C(P)$ the set of all covering relations of $P$.

A poset $P$ is called {\it bounded} if it has the unique minimal element and the unique maximal element which are usually denoted $\hat{0}$ and $\hat{1}$.

An (increasing) {\em chain} in $P$ is a sequence of distinct elements such that $x=x_1 < x_2 < \cdots < x_{n-1} < x_n = y$.
A chain in a poset $P$ is called {\em saturated}, if it is of the form $x=x_1 \leftarrow x_2 \leftarrow \cdots \leftarrow x_{n-1} \leftarrow x_n = y$.
A saturated chain in an interval $[x,y]$ is called {\em maximal}, if the end points of the chain are $x$ and $y$.
Recall also that a poset is called {\em graded} if all maximal chains between any two comparable elements $x \leq y$
have the same length. This amounts to the existence of an integer valued function $\ell_P : P\rightarrow \N$ satisfying
\begin{enumerate}
\item $\ell_P (\hat{0}) = 0$,
\item $\ell_P (y) = \ell_P(x) +1$ whenever $y$ covers $x$ in $P$.
\end{enumerate}
$\ell_P$ is called the {\em length function} of $P$.
In this case, the length of the interval $[\hat{0},\hat{1}]=P$ is called the {\em length} of the poset $P$.
For $i\in \N$, the {\em $i$-th level} of a graded poset $P$ is defined to be the subset $S\subset P$
consisting of elements of length $i$.
Thus, if $p_i$ denotes the number of elements of the $i$-th level of $P$, then
the length-generating function of $P$ is equal to $G_P(q) = \sum_{i\geq 0} p_i q^i$.

\vspace{.5cm}

We recalled the definition of the notion of EL-shellability in the previous section.
There is an elementary but a useful criterion to see if a poset is EL-shellable or not:
\begin{Lemma}[Proposition 3.1~\cite{DS}]\label{L1}
Let $P$ be a finite graded poset with the smallest and the largest elements, denoted by $\hat{0}$ and $\hat{1}$, respectively.
(Assume $\hat{0} \neq \hat{1}$.) Let $g: C(P) \rightarrow \varGamma$ be an EL-labeling of $P$.
Let $Q\subseteq P$ contain $\hat{0}$ and also a maximal element $z \neq \hat{0}$ (in the induced order).
Assume that $Q$ satisfies the following property: For all $x< y$, the unique rising chain from $x$ to $y$ in $P$ lies in $Q$.
Then $Q$ (with the induced order) is a graded poset with the same rank function as $P$ and $g$ restricted to $C(Q)$
is an EL-labeling for $Q$.
\end{Lemma}

\vspace{.5cm}

To check the EL-shellability of the posets of our paper,
we need to have a concrete way for comparing given two permutations in the Bruhat-Chevalley ordering:
For an integer valued vector $a= [a_1,\dots , a_n] \in \Z^n$, let
$\widetilde{a} = [a_{\alpha_1}, \dots , a_{\alpha_n}]$  be the
rearrangement of the entries $a_1,\dots,a_n$ of $a$ in a non-increasing fashion;
\begin{equation*}
a_{\alpha_1} \geq a_{\alpha_2} \geq \cdots \geq a_{\alpha_n}.
\end{equation*}
The \textit{containment ordering}, ``$\leq_c$,'' on $\Z^n$ is then defined by
\begin{equation*}
a=[a_1,\dots ,a_n] \leq_c b=[b_1,\dots ,b_n] \iff a_{\alpha_j} \leq
b_{\alpha_j}\ \text{for all}\  j=1,\dots,n,
\end{equation*}
where $\widetilde{a} = [a_{\alpha_1},\dots ,a_{\alpha_n}]$, and
$\widetilde{b} = [b_{\alpha_1},\dots,b_{\alpha_n}]$.
\begin{Example}
Let $x=[4,5,0,3,1]$, and let $y=[4,2,5,5,1]$. Then $x \leq_c y$, because
$\widetilde{x}=[5, 4, 3, 1, 0]$ and $\widetilde{y}= [5, 5, 4, 2, 1]$.
\end{Example}

For $k\in [n]$, the \textit{$k$-th truncation} $a[k]$ of $a=[a_1,\dots ,a_n]$ is defined to be $a[k]:=[a_1, a_2,\dots ,a_k]$.
A proof of the following lemma can be found in \cite{CR}.
\begin{Lemma}\label{L3}
Let $v=[v_1,\dots , v_n]$ and $w=[w_1,\dots , w_n]$ be two permutations from $S_n$.
Then $v\leq w$ (in Bruhat-Chevalley ordering) if and only if
\begin{equation*}
 \widetilde{v[k]} \leq_c \widetilde{w[k]}\ \text{for all}\ k=1,\dots,n.
\end{equation*}
\end{Lemma}
\begin{Example}
Let $x=[4,1,2,3,5]$, and let $y=[4,3,2,5,1]$. Then $x \leq y$, because
\begin{eqnarray*}
\widetilde{x[1]}=[4] &\leq_c& \widetilde{y[1]}=[4] ,\\
\widetilde{x[2]}=[4,1] &\leq_c& \widetilde{y[2]}=[4,3], \\
\widetilde{x[3]}=[4,2,1] &\leq_c& \widetilde{y[3]}=[4,3,2], \\
\widetilde{x[4]}=[4,3,2,1] &\leq_c& \widetilde{y[4]}=[5,4,3,2], \\
\widetilde{x[5]}=[5,4,3,2,1] &\leq_c& \widetilde{y[5]}=[5,4,3,2,1].
\end{eqnarray*}
\end{Example}

The EL-shellability of the Bruhat-Chevalley order on symmetric group is first proved by Edelman in~\cite{E}.
His EL-labeling is as follows: If $\sigma$ covers $\pi$ and the numbers $i$ and $j$, $i<j$, are interchanged in $\pi$ to get $\sigma$,
then the covering relation (the edge in the Hasse diagram) between $\pi$ and $\sigma$ is labeled by $(i,j)$.
It means that the covering relation between $\pi$ and $\sigma$ is labeled by $(i,j)$ when $\sigma=(i,j)\cdot\pi$,
where by $(i,j)$ we mean the transposition which interchanges $i$ and $j$. Edelman proves in~\cite{E} that this
labeling is indeed an EL-labeling. We use another EL-labeling of the Bruhat poset of $S_n$: the covering relation between
$\pi$ and $\sigma$ is labeled by $(i,j)$ when $\sigma=\pi\cdot (i,j)$. This is also an EL-labeling since the map
$\delta\mapsto\delta^{-1}$ is order-preserving and $\sigma=\pi\cdot (i,j)$ if and only if $\sigma^{-1}=(i,j)\cdot \pi^{-1}$.

\section{Gradedness and EL-shellability}\label{S:three}

Observe that a maximal set on which $\Omega$ is injective is the set of standard cyclic forms
which suit a certain composition of $n$.
Indeed, when we know the composition type we can put the parentheses on the one-line notation of the permutation
in the image of $\Omega$, and thus, we reconstruct its pre-image.

\vspace{.25cm}
\noindent
\textbf{Theorems A. \& C.}
Let $\lambda=(\lambda_1,\dots, \lambda_k)$ be a composition of $n$.
Then $C_\lambda$ is a bounded, graded, EL-shellable subposet of the Bruhat-Chevalley poset of $S_n$.
\vspace{.25cm}

\begin{proof}
The unique minimal element of $C_\lambda$ is the identity permutation $id=(1,2,\dots, n)$.
The unique maximal element $\omega_\lambda$ of $C_\lambda$ is obtained as follows.
Let $\pi_\lambda = T_1 \dots T_k$ be the permutation given in standard cycle form such that
\begin{align}\label{A:max element}
T_i =
\begin{cases}
(i) & \text{if } \lambda_i = 1, \\
 (i,n-\sum_{j=1}^{i-1} \lambda_j - (i-1), n-\sum_{j=1}^{i-1} \lambda_j - (i-2), \dots, n-\sum_{j=1}^{i-2} \lambda_j )
 & \text{if } \lambda_i > 1,
\end{cases}
\end{align}
where $T_i$ is a cycle of length $\lambda_i$.
For example, if $\lambda=(4,2,3,5)$, then
$$
\pi_{(4,2,3,5)} =(1,14,13,12)(2,11)(3,10,9)(4,8,7,6,5).
$$
By using Lemma~\ref{L3}, it is easy to verify that $\Omega(\pi_\lambda)$ is the maximal element of $C_\lambda$.
So, the poset $C_\lambda$ is bounded.

In order to prove that $C_\lambda$ is graded and EL-shellable we use Lemma \ref{L1}.
Thus, we must take two permutations $\delta, \tau\in C_\lambda$ such that $\delta<\tau$ in the Bruhat-Chevalley order of $S_n$,
and show that unique increasing chain (which is lexicographically smallest) between $\delta$ and $\tau$ lies completely inside $C_\lambda$.
Let $C_{\lambda}([\delta,\tau])$ be an interval in $C_\lambda$ with the smallest and the largest elements
$\delta$ and $\tau$, respectively.
In particular,
\begin{align}\label{standard form of delta}
\delta = \Omega(T_1T_2\cdots T_p),
\end{align}
where $T_i$ is a cycle of length $\lambda_i$,
and $1=\text{min} (T_1) < \text{min}(T_2) <\cdots < \text{min}(T_p)$ and furthermore the first entry in $T_i$ is
the smallest among all entries of $T_i$.

We know that $\delta$ is less than or equal to $\tau$ in the Bruhat poset of $S_n$, and denote by $B( [\delta, \tau] )$ the corresponding
interval in $S_n$. Note that, in general, $B([\delta, \tau])$ is not equal to $C_\lambda ([\delta, \tau])$.

{\em Claim:} Let $\mathfrak{c}: \delta = x_0 < x = x_1 < \dots < x_m = \tau$ denote the lexicographically smallest increasing chain in
$B([\delta, \tau])$. Then $\mathfrak{c}$ is contained in $C_\lambda ([\delta, \tau])$.

{\em Proof of the Claim:} We use induction on the length of the chain $\mathfrak{c}$. It is clear that if $m=1$, then there is nothing to prove.
So, we assume that $m > 1$. Now, by using the induction step, it suffices to show that $x=x_1$ of $\mathfrak{c}$ lies in $C_\lambda ([\delta,\tau])$.

Note that the label of the covering relation $x\rightarrow \delta$ is minimal among all coverings of $\delta$ in $B([\delta,\tau])$.
Let $\sigma = (k,s)$ denote the corresponding transposition that gives the minimal label. Hence, $x= \delta \cdot \sigma$ and $\ell(x) = \ell(\delta) +1$.
For convenience let us use one-line notation $\delta_1\delta_2\dots \delta_k \dots \delta_s \dots \delta_n$ to denote $\delta$.
Then $x=\delta_1\dots \delta_s \dots \delta_k\dots \delta_n$.
Also, it follows from the definition of the Bruhat-Chevalley ordering and of the EL-labeling (see Section \ref{S:preliminaries}) that
\begin{itemize}
\item[(i)] $\delta_k < \delta_s$,
\item[(ii)] for $k< i < s$, either $\delta_i < \delta_k$ or $\delta_i > \delta_s$,
\item[(iii)] $k$ is the smallest index with $\tau_k \neq \delta_k$.
\end{itemize}

Let $T_1\dots T_p$ be as in (\ref{standard form of delta}) and
let $1\leq j \leq l \leq p$ be such that $\delta_k \in \text{support}(T_j)$ and $\delta_s \in \text{support}(T_l)$.

We first assume that $j<l$.
Let $T_j'$ denote the cycle obtained from $T_j$ by replacing $\delta_k$ by $\delta_s$, and
let $T_l'$ denote the cycle obtained from $T_l$ by replacing $\delta_s$ by $\delta_k$.
Clearly, if we can show that $T_1\dots T_j' \dots T_l' \dots T_p$ is in standard cycle form, then the proof follows in this case, because
$\Omega(T_1\dots T_j' \dots T_l' \dots T_p) = x$.
Let $T_q$ be a cycle such that $j \leq q \leq l$. We do our analysis in two steps.
First, we consider when $j< q \leq l$.
Then (ii) implies that $\text{min}(T_q) < \delta_s$. In this case, by using (ii) once more,
we see that $\text{min}(T_q) < \delta_k$. Therefore, interchanging $\delta_k$ and $\delta_s$ does not alter the composition type,
hence we are done in this case.
The second case, when $j=q$ splits into two sub-cases; either
$\text{min}(T_j)= \delta_k$, or $\text{min}(T_j) < \delta_k$. The latter sub-case is the same as our first step.
For the former sub-case, by using (ii) one more time, we see that if $j \leq q' \leq l$, then any entry of $T_{q'}$,
except $\delta_k$, satisfies $\delta_s \leq \text{min}(T_{q'})$. Therefore, interchanging $\delta_s$ and $\delta_k$
does not alter the composition type in this sub-case, neither.

Finally, we assume that $j=l$. If $\delta_k$ is not the first entry of $T_j$, then we denote by $T_j'$ the
cycle that is obtained from $T_j$ by interchanging $\delta_k$ and $\delta_s$. It is clear that the permutation
$T_1\dots T_j'\dots T_p$ is in standard form and applying $\Omega$ to it gives us $x$.
We finish our proof by eliminating the possibility $\delta_k = \text{min}(T_j)$.
To this end, let $\pi = R_1\dots R_j \dots R_n$ denote a permutation given in standard cycle notation such that $\Omega(\pi)=\tau$.
Hence, $R_i$ is a cycle of length $\lambda_i$,
and $1=\text{min} (R_1) < \text{min}(R_2) <\cdots < \text{min}(R_p)$ and furthermore the first entry in $R_i$ is
the smallest among all entries of $R_i$. In addition, by (iii) we know that $T_i = R_i$ for $i < j$.
Now, because $\delta_k$ is the smallest element of $T_j$, it cannot occur in $R_j$ for $\text{min}(R_j)=\tau_k$.
If $R_q$, $q>j$ is the cycle of $\pi$ containing $\delta_k$, then $\text{min}(R_q) < \tau_k$. But this contradicts with the fact that
$\text{min}(R_j)=\tau_k < \text{min}(R_q)$, hence the proof is complete.
\end{proof}

\begin{Proposition}
For $\lambda = (\lambda_1,\dots, \lambda_k) \vDash n$, the length of $C_\lambda$ is equal to
${n-1 \choose 2} +k-1- \sum_{r=1}^{k} (r-1) \lambda_r$.
\end{Proposition}
\begin{proof}
It is enough to compute the length of the maximal element $\omega_\lambda$ in $C_\lambda$:
\begin{align*}
inv(\omega_\lambda) &= \sum_{i=1}^{n-2} i - \sum_{j=1}^{k-1} \left( n- \sum_{r=1}^j \lambda_r -1 \right) \\
&= {n-1 \choose 2} - (n-1)(k-1) + \sum_{r=1}^{k-1} (k-r) \lambda_r \\
&= {n-1 \choose 2} - (n-1)(k-1) + \sum_{r=1}^{k-1} k\lambda_r -\sum_{r=1}^{k-1} r\lambda_r \\
&= {n-1 \choose 2} - (n-1)(k-1) + \sum_{r=1}^{k} k\lambda_r -\sum_{r=1}^{k} r\lambda_r \\
&= {n-1 \choose 2} - (n-1)(k-1) + nk  -\sum_{r=1}^{k} (r-1) \lambda_r -n\\
&= {n-1 \choose 2} +k-1- \sum_{r=1}^{k} (r-1) \lambda_r.
\end{align*}
\end{proof}

\section{Unimodality and rank symmetry}\label{S:four}

Let $Comp(n)$ denote the set of all compositions of $n$. Define the operator $\mc{ST}: Comp(n) \rightarrow Comp(n)$ as follows:
For $\lambda =\left(\lambda_1,\dots,\lambda_k\right)\in Comp(n)$, $\mathcal{ST}(\lambda)$ is
the composition of $n$ obtained from $\lambda$ by splitting the part $\lambda_j$ into two parts $\lambda_j-1,1$,
where $\lambda_j$ is the rightmost part of $\lambda$ which is greater than $1$.
For example, $\mathcal{ST}(4,3)=(4,2,1)$ and $\mathcal{ST}(1,2,1,3,1,1)=(1,2,1,2,1,1,1)$.
Notice that for any composition $\lambda$ if we keep applying the operation $\mathcal{ST}$, then
eventually we arrive at the composition with all parts equal to 1.

Our first important observation regarding the relationship between $C_\lambda$ and $C_{\mc{ST}(\lambda)}$ is that
if $\lambda_k>1$, then $C_\lambda=C_{\mathcal{ST}(\lambda)}$. Indeed, this is easy to verify and does not need a proof.
If, on the contrary, $\lambda_k=1$, then there is a natural embedding of $C_{\mc{ST}(\lambda)}$ in $C_\lambda$.
Moreover, as we show below, $C_\lambda$ is a union of several copies of $C_{\mathcal{ST}(\lambda)}$ glued together in a certain way.
Also, as we are going to explain in the sequel, it follows from these observations that the poset $C_\lambda$ is
rank-symmetric and unimodal. In other words, if $r_i$ denotes the number of elements of length $i$ in $C_\lambda$ and $M$ denotes
the length of $\omega_\lambda$, the maximal element of $C_\lambda$, then
\begin{itemize}
\item $r_i = r_{M-i}$, for all $i=0,\dots, \lfloor M/2 \rfloor$,
\item $r_0 \leq r_1 \leq \cdots \leq r_{\lfloor M/2 \rfloor } \geq r_{\lfloor M/2 \rfloor +1 } \geq \cdots \geq r_M$.
\end{itemize}

Let us begin with exploring in detail the example of $C_{(4,2)}$, whose Hasse diagram is depicted in Figure~\ref{F}, below.
Our statements about this example are going to stay valid for a general composition $\lambda \vDash n$.
First of all, we know that $C_{(4,2)}=C_{(4,1,1)}$. By definition, $\mathcal{ST}(4,1,1)=(3,1,1,1)$. Notice from Figure~\ref{F} that
$C_{(4,1,1)}$ is a union of three copies of $C_{(3,1,1,1)}$, shifted by one-level each.
\begin{figure}[htp]
\begin{center}
\scalebox{0.7} 
{\includegraphics[height=10in, width=10in]{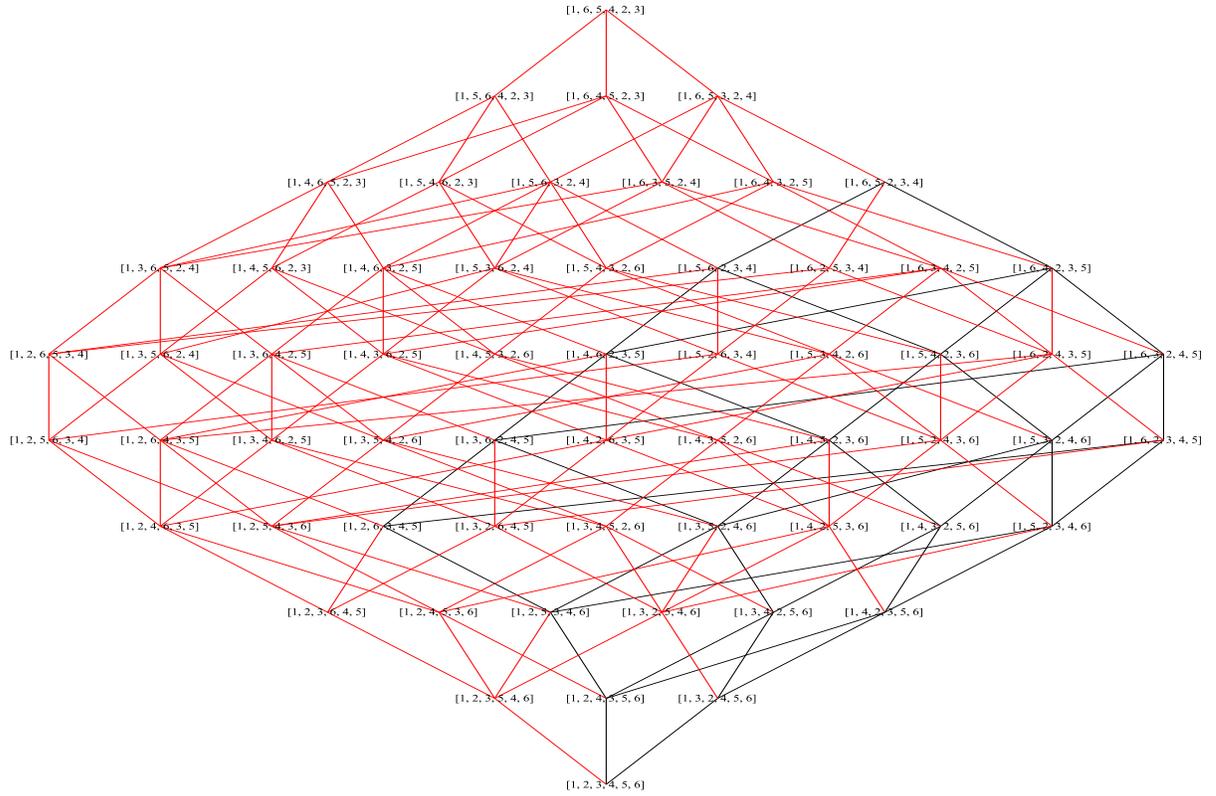}}
\caption{$C_{(3,1,1,1)}$ in $C_{(4,2)}=C_{(4,1,1)}$}
\label{F}
\end{center}
\end{figure}
Indeed, let us denote these posets by $C_{(3,1,1,1)}^{(0)}$, $C_{(3,1,1,1)}^{(1)}$, $C_{(3,1,1,1)}^{(2)}$.
Then the first copy $C_{(3,1,1,1)}^{(0)}$ is the most natural embedding of the poset $C_{(3,1,1,1)}$ into $C_{(4,1,1)}$;
its minimal element is the minimal element of $C_{(4,2)}$ (which is $id=[1,2,3,4,5,6]$), and its maximal element is $[1,6,5,2,3,4]$.
The Hasse diagram of $C_{(3,1,1,1)}^{(0)}$ in that of $C_{(4,1,1)}$ is depicted by using the black edges as in Figure~\ref{F}.

Next, we describe $C_{(3,1,1,1)}^{(1)}$. Notice that for any $\pi^{(0)}\in C_{(3,1,1,1)}^{(0)} = C_{(3,1,1,1)}$
the fourth entry of $\pi^{(0)}$ is less than its fifth entry (in one-line notation, of course).
Set $\pi^{(1)} :=[1,\pi^{(0)}_2,\pi^{(0)}_3,\pi^{(0)}_5,\pi^{(0)}_4,\pi^{(0)}_6]$ and observe that $\pi^{(1)}$ lies in $C_{(4,1,1)}$.
Observe also that $\pi^{(1)}$ is uniquely determined by $\pi^{(0)}$.
For example, check that $[1,6,5,2,3,4]^{(1)}=[1,6,5,3,2,4]$.
Following this scheme, we set $C_{(3,1,1,1)}^{(1)} := \{ \pi^{(1)} \in C_{(4,1,1)}:\ \pi^{(0)} \in C_{(3,1,1,1)}^{(0)}\}$.
It is obvious that for every $\pi^{(0)}\in C_{(3,1,1,1)}^{(0)}$ we have $inv(\pi^{(1)})=inv(\pi^{(0)})+1$,
and moreover, $\pi \leq \tau$ in $C_{(3,1,1,1)}^{(0)}$ if and only if $\pi^{(1)} \leq \tau^{(1)}$ in $C_{(3,1,1,1)}^{(1)}$.
The poset $C_{(3,1,1,1)}^{(2)}$ is defined similarly. We observe that for each $\pi^{(1)}\in C_{(3,1,1,1)}^{(1)}$,
$\pi^{(1)}_4<\pi^{(1)}_6$. Thus, we interchange the fourth and the sixth entries to obtain
$$
\pi^{(2)}:=[1,\pi^{(1)}_2,\pi^{(1)}_3,\pi^{(1)}_6,\pi^{(1)}_5,\pi^{(1)}_4]=[1,\pi^{(0)}_2,\pi^{(0)}_3,\pi^{(0)}_6,\pi^{(0)}_4,\pi^{(0)}_5]\in C_{(3,1,1,1)}^{(2)}.
$$
It is easy to verify that $inv(\pi^{(2)})=inv(\pi^{(1)})+1=inv(\pi^{(0)})+2$ and $\pi^{(1)} \leq \tau^{(1)}$ in $C_{(3,1,1,1)}^{(1)}$ if and only if
$\pi^{(2)} \leq \tau^{(2)}$ in $C_{(3,1,1,1)}^{(2)}$.
Finally, we observe that $C_{(3,1,1,1)}^{(0)}$, $C_{(3,1,1,1)}^{(1)}$, $C_{(3,1,1,1)}^{(2)}$
are pairwise disjoint and their union is equal to $C_{(4,1,1)}$.

Next, we analyze $C_{(3,1,1,1)}$ by iterating the above scheme.
Notice that $\mathcal{ST}(3,1,1,1)=(2,1,1,1,1)$, and that
$C_{(2,1,1,1,1)}$ is the chain
$$
[1,2,3,4,5,6]<[1,3,2,4,5,6]<[1,4,2,3,5,6]<[1,5,2,3,4,6]<[1,6,2,3,4,5].
$$
It turns out that the poset $C_{(3,1,1,1)}^{(0)}=C_{(3,1,1,1)}$ consists of four isomorphic copies
of $C_{(2,1,1,1,1)}$. Indeed, the poset $C_{(2,1,1,1,1)}^{(0)}$ is exactly $C_{(2,1,1,1,1)}$, and
the poset $C_{(2,1,1,1,1)}^{(1)}$ is equal to
$$
[1,2,4,3,5,6]<[1,3,4,2,5,6]<[1,4,3,2,5,6]<[1,5,3,2,4,6]<[1,6,3,2,4,5].
$$
The poset $C_{(2,1,1,1,1)}^{(2)}$ is
$$
[1,2,5,3,4,6]<[1,3,5,2,4,6]<[1,4,5,2,3,6]<[1,5,4,2,3,6]<[1,6,4,2,3,5],
$$
and finally, the poset $C_{(2,1,1,1,1)}^{(3)}$ is the chain
$$
[1,2,6,3,4,5]<[1,3,6,2,4,5]<[1,4,6,2,3,5]<[1,5,6,2,3,4]<[1,6,5,2,3,4].
$$
The union of these four chains form the poset $C_{(3,1,1,1)}^{(0)}$.
Of course, both of the posets $C_{(3,1,1,1)}^{(1)}$, and $C_{(3,1,1,1)}^{(2)}$ have similar decompositions into
corresponding chains.

Finally, we look at situation $\mathcal{ST}(2,1,1,1,1)=(1,1,1,1,1,1)$.
The poset $C_{(1,1,1,1,1,1)}$ is a singleton, and $C_{(2,1,1,1,1)}$ is the chain of five elements.
Thus, $C_{(2,1,1,1,1)}$ consists of five copies $C_{(1,1,1,1,1,1)}^{(0)}$,\dots,$C_{(1,1,1,1,1,1)}^{(4)}$ of $C_{(1,1,1,1,1,1)}$.

Before we explain how the above described decompositions imply the unimodality and the rank-symmetry of $C_{(4,1,1)}$,
let us point out that although $C_{(4,1,1)}$ is a disjoint union of three shifted copies of $C_{(3,1,1,1)}$, the poset $C_{(4,1,1)}$
is not isomorphic to a Cartesian product of a chain with 3 elements and the subposet $C_{(3,1,1,1)}$.

The five one-element posets $C_{(1,1,1,1,1)}^{(i)}$, $i=0,\dots,4$ form the chain $C_{(2,1,1,1,1)}$.
Four copies of the poset $C_{(2,1,1,1,1)}$ form $C_{(3,1,1,1)}$. Finally, there are three copies of $C_{(3,1,1,1)}$ forming $C_{(4,1,1)}$.

We paraphrase pictorially starting from $C_{(2,1,1,1,1)}$:
 $$
\begin{matrix}
1 &\, &\, &\, &\,\\
\, &1 &\, &\, &\,\\
\, &\, &1 &\, &\,\\
\, &\, &\, &1 &\,\\
\, &\, &\, &\, &1\\
1 &1   &1  &1  &1
\end{matrix}
$$
Here, we read the rows of the array starting from top towards bottom. Each of the first five rows corresponds to a copy of $C_{(1,1,1,1,1,1)}$.
Each of them has the unique entry 1 since these come from one-element posets.
Each next row is shifted one unit towards right because each copy of $C_{(1,1,1,1,1,1)}$ is
placed one level above the previous copy in $C_{(2,1,1,1,1)}$.
The sixth row is the sequence of numbers of elements at each level of $C_{(2,1,1,1,1)}$.

We repeat the same process for the four copies of $C_{(2,1,1,1,1)}$ in $C_{(3,1,1,1)}$:
$$
\begin{matrix}
1 &1 &1 &1 &1 &\, &\, &\,\\
\,&1 &1 &1 &1 &1 &\, &\, \\
\, &\, &1  &1 &1 &1  &1 &\,\\
\, &\, &\, &1  &1 &1 &1  &1\\
1 &2 &3 &4 &4 &3 &2 &1
\end{matrix}
$$
The top first four rows are sequences of numbers of elements at each level of $C_{(2,1,1,1,1)}$.
There is a shifting because each next copy starts one level above the previous copy of $C_{(2,1,1,1,1)}$ in $C_{(3,1,1,1)}$.
The last row is the sequence of the cardinalities of the levels of $C_{(3,1,1,1)}$.
Clearly, each entry of the sequence is the sum of the terms lying above in its column.

In the same way we display what happens when three copies of the poset $C_{(3,1,1,1)}$ form the poset $C_{(4,1,1)}$:
$$
\begin{matrix}
1 &2 &3 &4 &4 &3 &2 &1 &\, &\,\\
\,&1 &2 &3 &4 &4 &3 &2 &1  &\,\\
\,&\,&1 &2 &3 &4 &4 &3 &2  &1 \\
1 &3 &6 &9 &11 &11 &9 &6 &3 &1
\end{matrix}
$$

More generally, given $a_1$, $a_2$,\dots, $a_{t-1}$, $a_t$, $a_{t-1}$,\dots, $a_2$, $a_1$, a symmetric and unimodal sequence,
the column sums of the diagram
$$
\begin{matrix}
a_1 & a_2 & \cdots  & a_{t}   & \cdots  & a_2   & a_1   & & &  \\
\, &a_1  &a_2    &\cdots   &a_{t}    &\cdots  &a_2    &a_1   & & \\
\, &\,   &a_1     &a_2    &\cdots   &a_{t}     &\cdots &a_2    &a_1 & \\
\, &\,   &\,     &\ddots  &\ddots  &\ddots  &\ddots  &\ddots   &\ddots &  \ddots  
\end{matrix}
$$
give the sequence $a_1$, $a_1+a_2$, $a_1+a_2+a_3$,\dots, $a_1+a_2+a_3$, $a_1+a_2$, $a_1$, which is obviously
symmetric and unimodal, also.

Using these observations we easily compute $\sum_{\pi\in C_{(4,1,1)}}q^{inv(\pi)}$:
The length-generating function of the chain of five elements $C_{(2,1,1,1,1)}$ is just $1+q+q^2+q^3+q^4=[5]_q$.
Four copies of $C_{(2,1,1,1,1)}$ form $C_{(3,1,1,1)}$ in the way that we described, so the length-generating function of
$C_{(3,1,1,1)}$ is $[4]_q[5]_q$, which is equal to $1+2q+3q^2+4q^3+4q^4+3q^5+2q^6+q^7$.
Three copies of $C_{(3,1,1,1)}$ form $C_{(4,1,1)}$ in the way that we described, so the length-generating function of $C_{(4,1,1)}$ is
$$
\sum_{\pi\in C_{(4,1,1)}}q^{inv(\pi)}=[3]_q[4]_q[5]_q=1+3q+6q^2+9q^3+11q^4+11q^5+9q^6+6q^7+3q^8+q^9\,.
$$

We are ready to deal with the general case.
\begin{Lemma}
Let $\lambda=\left(\lambda_1,\dots,\lambda_k\right) \vDash n$ be a composition of $n$,
and let $\lambda_j$ be the rightmost part of $\lambda$ which is greater than $1$. Then
$$
C_\lambda=\bigcup_{i=0}^{k-j}C_{\mathcal{ST}(\lambda)}^{(i)},
$$
where $C_{\mathcal{ST}(\lambda)}^{(i)} \subseteq C_\lambda$, $i=1,\dots, k-j$  is a subposet isomorphic to $C_{\mathcal{ST}(\lambda)}$,
and defined as follows:
$C_{\mathcal{ST}(\lambda)}^{(0)}$ is just the natural embedding of $C_{\mathcal{ST}(\lambda)}$ into $C_\lambda$.
To build $C_{\mathcal{ST}(\lambda)}^{(1)}$ from $C_{\mathcal{ST}(\lambda)}^{(0)}$,
pick an element $\pi^{(0)}\in C_{\mathcal{ST}(\lambda)}^{(0)}$ and notice that $\pi^{(0)}_t<\pi^{(0)}_{t+1}$,
where $t=\lambda_1+\lambda_2+\cdots +\lambda_j$. Define $\pi_{(1)}\in C_{\mathcal{ST}(\lambda)}^{(1)}$
by interchanging $\pi^{(0)}_t$ and $\pi^{(0)}_{t+1}$ in the one-line notation of $\pi^{(0)}$, i.e.,
$\pi^{(1)}=[1,\pi^{(0)}_2,\dots,\pi^{(0)}_{t-1},\pi^{(0)}_{t+1},\pi^{(0)}_{t},\dots,\pi^{(0)}_{n}]$.
To build $C_{\mathcal{ST}(\lambda)}^{(2)}$ from $C_{\mathcal{ST}(\lambda)}^{(1)}$, we start with an
element $\pi^{(1)}$ of $C_{\mathcal{ST}(\lambda)}^{(1)}$. Then
the permutation $\pi^{(2)}$ is defined by interchanging $t$-th and $(t+2)$-th entries in the one-line notation of $\pi^{(1)}$, and so on.
\end{Lemma}

\begin{proof}
The inclusion $C_\lambda\supseteq\bigcup_{i=0}^{k-j}C_{\mathcal{ST}(\lambda)}^{(i)}$ is
obvious since each permutation in each $C_{\mathcal{ST}(\lambda)}^{(i)}$  belongs to $C_\lambda$ by definition.
Thus, it suffices to prove the inclusion $C_\lambda\subseteq\bigcup_{i=0}^{k-j}C_{\mathcal{ST}(\lambda)}^{(i)}$.

Let $m(\lambda)$ be the number of parts of size $1$ in the right tail of the composition $\lambda$.
For example, if $\lambda=(1,5,1,3,1,1,1,1)$, then $m(\lambda)=4$.
Then by the definition of $C_\lambda$ the numbers in rightmost $m(\lambda)$ positions of every permutation of $C_\lambda$ increase.
It is obvious that for any permutation of $C_{\mathcal{ST}(\lambda)}$ the numbers in $m(\lambda)+1$ rightmost positions increase.
Take a permutation $\pi\in C_\lambda$ and let $i$ be the number of inversions of the right tail of length $m(\lambda)+1$ of the permutation $\pi$.
Then by definition $\pi\in C_{\mathcal{ST}(\lambda)}^{(i)}$. This completes the proof of our claim.

To illustrate we look at our earlier example once more and show that
each $\pi\in C_{(4,1,1)}$ belongs to exactly one of the sets $C_{(3,1,1,1)}^{(0)}$, $C_{(3,1,1,1)}^{(1)}$, $C_{(3,1,1,1)}^{(2)}$.
For instance, take $\pi=[1,3,6,4,2,5]\in C_{(4,1,1)}$. We have $m(4,1,1)+1=3$. The right tail of $\pi$ of length $3$ is $[4,2,5]$.
There is only one inversion in $[4,2,5]$, and therefore, $\pi\in C_{(3,1,1,1)}^{(1)}$.
\end{proof}

We are ready to prove

\vspace{.25cm}
\noindent
\textbf{Theorems B.}
Let $\lambda=\left(\lambda_1,\dots,\lambda_k\right)$ be a composition of $n$.
The poset $C_\lambda$ is unimodal and rank-symmetric and moreover its length-generating function is equal to
$$
\sum_{\pi\in C_\lambda}q^{inv(\pi)}=[i_1]_q[i_2]_q\cdots[i_m]_q
$$
for a certain sequence $2\leqslant i_1<i_2<\cdots < i_m\leqslant n-1$.

\begin{proof}
The zero level of $C_\lambda$ consists only of $id$ which is the zero-level entry of $C_{\mathcal{ST}(\lambda)}^{(0)}$.
The level $1$ of $C_\lambda$ consists of all level $1$ elements of $C_{\mathcal{ST}(\lambda)}^{(0)}$
together with the zero-level element of $C_{\mathcal{ST}(\lambda)}^{(1)}$.

The level 2 of $C_\lambda$ consists of all level 2 elements of $C_{\mathcal{ST}(\lambda)}^{(0)}$
together with all level 1 elements of $C_{\mathcal{ST}(\lambda)}^{(1)}$ and the zero-level element of
$C_{\mathcal{ST}(\lambda)}^{(2)}$, and so on. Let $f(q)$ be the length generating function of
$C_{\mathcal{ST}(\lambda)}$. Let $\lambda_j$ be the rightmost part of $\lambda$
which is greater than $1$, so $C_\lambda$ is the union of $C_{\mathcal{ST}(\lambda)}^{(0)}$,
$C_{\mathcal{ST}(\lambda)}^{(1)}$, \dots , $C_{\mathcal{ST}(\lambda)}^{(k-j)}$. We immediately have
$$
\sum_{\pi\in C_\lambda}q^{inv(\pi)}=[k-j+1]_qf(q)\,.
$$
Continuing with applying the $\mathcal{ST}$ operation, we get at some step
$\mathcal{ST}(\mathcal{ST}(\cdots(\mathcal{ST}(\lambda))))=(2,1,1,\dots,1)$.
The poset $C_{(2,1,1,\dots,1)}$ is just the chain of $n-1$ elements and its length generating function is $[n-1]_q$.
So the length-generating function is $[k-j+1]_q[i_2]_q\cdots[n-1]_q$ as claimed.
It is well known that such a polynomial is palindromic and so the poset is unimodal and rank-symmetric.

\end{proof}

\section{The order complex of $C_\lambda$}\label{S:five}

The {\em M\"obius function} of a poset $P$ is defined recursively by the formula
\begin{align*}
\mu ([x,x]) &= 1 \\
\mu ([x,y]) &= - \sum_{x \leq z < y} \mu([x,z])
\end{align*}
for all $x \leq y $ in $P$.

Let $\hat{0}$ and $\hat{1}$ denote the smallest and the largest elements of $P$, respectively.
It is well known that $\mu(\hat{0},\hat{1})$ is equal to the ``reduced Euler characteristic'' $\widetilde{\chi}(\Delta(P))$
of the topological realization of the order complex of $P$. See Proposition 3.8.6 in \cite{EC1}.

Let $\varGamma$ denote a finite, totally ordered poset and let $g$ be a $\varGamma$-valued
function defined on $C(P)$. Then $g$ is called an {\em $R$-labeling}
for $P$, if for every interval $[x,y]$ in $P$, there exists a unique saturated chain
$x=x_1 \leftarrow x_2 \leftarrow \cdots \leftarrow x_{n-1} \leftarrow x_n = y$ such that
\begin{align}\label{R-label}
g(x_1,x_2) \leq g(x_2,x_3) \leq \cdots \leq g(x_{n-1},x_n).
\end{align}
Thus, $P$ is EL-shellable, if it has an $R$-labeling $g:C(P)\rightarrow \varGamma$ such that
for each interval $[x,y]$ in $P$ the sequence (\ref{R-label}) is lexicographically smallest among all sequences of the form
$$
(g(x,x_2'),g(x_2',x_3'), \dots , g(x_{k-1}',y)),
$$
where $x \leftarrow x_2 \leftarrow ' \cdots \leftarrow x_{k-1}' \leftarrow y$.

Suppose $P$ is of length $n\in \N$ with the length function $\ell = \ell_P: P\rightarrow \N$.
For $S\subseteq [n]$, let $P_S$ denote the subset $P_S = \{ x\in P:\ \ell(x) \in S\}$.
We denote by $\mu_S$ the M\"obius function of the poset $\hat{P}_S$ that is obtained from $P_S$ by
adjoining a smallest and a largest element, if they are missing.
Suppose also that $g:C(P)\rightarrow \varGamma$ is an $R$-labeling for $P$. In this case
it is well known that $(-1)^{|S|-1} \mu_S (\hat{0}_{\hat{P}_S},\hat{1}_{\hat{P}_S})$ is equal to the number of maximal chains
$x_0=\hat{0} \leftarrow x_1 \leftarrow \cdots \leftarrow x_n=\hat{1}$ in $P$ for which the descent set of the sequence
$(g(x_0,x_1),\dots,g(x_{n-1},x_n))$ is equal to $S$, or equivalently,
$\{ i \in [n]:\ g(x_{i-1},x_i) \geq g(x_{i+1},x_i) \} = S$. See Theorem 3.14.2 in \cite{EC1}.

\vspace{.25cm}
\noindent
\textbf{Theorems D.}
Let $\lambda=\left(\lambda_1,\dots,\lambda_k\right)$ be a composition of $n$.
\begin{enumerate}
\item If $\lambda=(n)$, or $\lambda=(n-1,1)$, then $C_\lambda$ is isomorphic to the Bruhat-Chevalley poset on $S_{n-1}$,
embedded into $S_n$ as the set of permutations fixing $1$.
Similarly, if $\lambda_1=\lambda_2=\cdots=\lambda_{k-1}=1$ and $\lambda_k=n-k+1$ or $\lambda_1=\lambda_2=\cdots=\lambda_{k-2}=1$,
$\lambda_{k-1}=n-k$, $\lambda_k=1$ (like $(1,1,1,4,1)$), then $C_\lambda$ is isomorphic to the Bruhat-Chevalley poset $S_{n-k}$
suitably embedded into $S_n$.
\item In all other cases the order complex of $C_\lambda$ triangulates a ball.
\end{enumerate}

\begin{proof}
The first statement is obvious.

Assume that $\lambda \vDash n$ is a composition different than those considered in the first part, and assume also that $\lambda_1>1$.
As before, denote by $\omega_\lambda$ the maximal element of $C_\lambda$.
We know from Section \ref{S:three} that the poset $C_\lambda$ is EL-shellable.
In order to show that its order complex triangulates a ball, by the discussion
at the beginning of this section it is enough to show that there does not exist a decreasing chain going from
$id$ to $\omega_\lambda$ so that $\mu_{C_\lambda}([id,\omega_\lambda])=0$.
We prove this by contradiction.

Assume contrary that there exists an increasing increasing chain starting at the maximal element $\omega_\lambda=[1,n,\dots,2,\dots]$
and going down to the minimum $id=[1,2,\dots,n]$. (Thus, it is a decreasing chain from bottom to top.)
Note that the number $2$ is at the second entry in $id$, and since by our hypothesis $\lambda_1>1$, in $\omega_\lambda$
the number $2$ is at the position $\lambda_1+1$ (which is greater than 2, of course).
As we go downwards in our increasing chain, in order to be able to move 2 to the second position
we need, at some step, to use an edge-label of the form $(2,*)$. Recall that these edge-labels means multiplying
from the right by the corresponding reflection. Since our chain is consists of covering relations,
as we move from top to bottom the inversion number of the permutation must decrease by one.
It follows that our labeling should start with with the label $(2,3)$, then use $(2,4)$, until we use $(2,\lambda_1+1)$,
which puts the number $2$ into the second position.
Notice that since we assumed that $\lambda$ is not of the form described in the first statement of our theorem,
the number $3$ must be in the position $\lambda_1+\lambda_2+1$ which is to the right of the number $2$ in
$\omega_\lambda=[1,n,\dots,2,\dots,3,\dots]$.
But this means that the permutation $\sigma=\omega_\lambda\cdot(2,3)\cdots(2,\lambda_1+1)\notin C_\lambda$. Indeed,
the transpositions we use up to this point does not change the position of the number $3$, for
they move only those numbers that are initially at the positions from the second to the $(\lambda_1+1)$-th, and
for $\sigma \in C_\lambda$ the number $3$ appear at the position $\lambda_1+1$ or smaller.
This completes the proof that there is no decreasing chain from the minimum to maximum in $C_\lambda$ when $\lambda_1>1$.
If $\lambda \vDash n$ has $m$ parts of size $1$ in its beginning, then we argue similarly,
by using the number $m+2$ instead of $2$, and $m+3$ instead of $3$.
This completes our proof.
\end{proof}

\section{Conclusion}\label{S:six}

This work can be continued in several directions.
For example, one can investigate the poset $\Omega(Conj(\nu))$, where $Conj(\nu)$ is the conjugacy class of the
symmetric group which corresponds to the partition $\nu$.
Obviously, $\Omega$ is not necessarily a one-to-one mapping when we consider it on $Conj(\nu)$ and
$\Omega(Conj(\nu))=\bigcup_{\lambda}C_\lambda$, where $\lambda$ runs over all compositions obtained
by permuting the parts of the partition $\nu$.
The computer calculations performed for few examples suggest that such posets are bounded, graded and unimodal.
However, they are not necessarily rank-symmetric.
The natural question is to whether they are EL-shellable, and determining their order complexes.

Another possibly interesting direction is to understand whether the posets presented here are related in some nice combinatorial way to representations of $S_n$ or of the type A 0-Hecke algebra along the lines of~\cite{M}.

\end{document}